\documentclass[a4paper,11pt,final]{scrartcl}
\usepackage[utf8]{inputenc}
\usepackage{ifpdf}
\usepackage{latexsym}
\usepackage{exscale}
\usepackage{amsfonts}
\usepackage{eucal}
\usepackage{cmmib57}
\usepackage[centertags]{amsmath}
\usepackage{amssymb}
\usepackage{amsbsy}
\usepackage{amstext}
\usepackage{amsthm,thmtools}
\usepackage{xcolor}
\usepackage{framed}
\usepackage{tocbibind}
\usepackage{tocloft}
\usepackage{etoolbox}
\usepackage[ruled,vlined,algosection]{algorithm2e}
\usepackage{delarray}
\usepackage[automark]{scrpage2}
\usepackage{times}
\usepackage[toc,page]{appendix}
\usepackage{textcomp}
\usepackage{mflogo}
\usepackage{here}
\usepackage[round]{natbib}
\usepackage{here}
\usepackage{longtable}
\usepackage{threeparttablex}
\usepackage{multirow}
\usepackage{multicol}
\usepackage{caption}
\usepackage{lscape}
\usepackage{rotating,capt-of}
\usepackage[margin=15mm]{geometry}
\usepackage{pgf}
\usepackage{tikz}
\usetikzlibrary{arrows,automata,shapes}
\usepackage{verbatim}
\newcounter{treeline}

\ifvtexpdf\pdftrue\fi
\ifpdf
\usepackage{graphicx}
\usepackage[pdftex]{thumbpdf}
\usepackage{nameref}
\usepackage[pdftex,colorlinks]{hyperref}
\hypersetup{
           breaklinks=true,
           pdfpagemode=UseOutlines,
           colorlinks=true,
           linkcolor=blue,
           citecolor=blue,
           bookmarksopen=true,
           pdftitle={Simplifying the Kohlberg Criterion on the Nucleolus: A Correct Approach},
           pdfauthor={Holger Ingmar Meinhardt},
           pdfsubject={Tu Games, Propositional Logic},
           pdfkeywords={Kohlberg Criteria, Nucleolus, Balancedness, Propositional Logic, Material Implication, Indirect Proof, Tu Games},
}

\else
\usepackage{epsfig}
\usepackage[ps2pdf]{thumbpdf}
\usepackage{nameref}
\usepackage[ps2pdf]{hyperref}

\hypersetup{
           breaklinks=true,
           pdfpagemode=UseOutlines,
           colorlinks=true,
           linkcolor=blue,
           citecolor=blue,
           bookmarksopen=true,
           pdftitle={Simplifying the Kohlberg Criterion on the Nucleolus: A Correct Approach},
           pdfauthor={Holger Ingmar Meinhardt},
           pdfsubject={Tu Games, Propositional Logic},
           pdfkeywords={Kohlberg Criteria, Nucleolus, Balancedness, Propositional Logic, Material Implication, Indirect Proof, Tu Games},
}
\fi

\addtocounter{MaxMatrixCols}{20}

\makeatletter
\renewcommand{\section}{\@startsection
  {section}%
  {1}%
  {0em}%
  {-\baselineskip}%
  {0.5\baselineskip}%
  {\centering\normalfont\Large\scshape\mdseries}}%
\makeatother

\makeatletter
\renewcommand{\subsection}{\@startsection
  {subsection}%
  {2}%
  {0em}%
  {-\baselineskip}%
  {0.5\baselineskip}%
  {\normalfont\large\scshape\mdseries}}%
\makeatother

\makeatletter
\renewcommand*\env@matrix[1][c]{\hskip -\arraycolsep
  \let\@ifnextchar\new@ifnextchar
  \array{*\c@MaxMatrixCols #1}}
\makeatother

\makeatletter

\newenvironment{theopargself*}
    {\def\@spopargbegintheorem##1##2##3##4##5{\trivlist
         \item[\hskip\labelsep{##4##1\ ##2}]{\hspace*{-\labelsep}##4##3\@thmcounterend}##5}
     \def\@Opargbegintheorem##1##2##3##4{##4\trivlist
         \item[\hskip\labelsep{##3##1}]{\hspace*{-\labelsep}##3##2\@thmcounterend}}}{}
\def \@floatboxreset {%
        \reset@font
        \small
        \@setnobreak
        \@setminipage
}
\def\figure{\@float{figure}}
\@namedef{figure*}{\@dblfloat{figure}}
\def\table{\@float{table}}
\@namedef{table*}{\@dblfloat{table}}
\def\fps@figure{htbp}
\def\fps@table{htbp}

\makeatother

\DeclareMathOperator*{\argmax}{\arg\!\max}

\makeatletter
\@addtoreset{equation}{section}
\makeatother
\makeatletter
\@addtoreset{table}{section}
\makeatother

\theoremstyle{plain}
\newtheorem{theorem}{Theorem}[section]

\newtheorem{definition}{Definition}[section]
\newtheorem{chara}{Characterization}[section]

\newtheoremstyle{break}% name
  {9pt}%      Space above, empty = `usual value'
  {9pt}%      Space below
  {\itshape}% Body font
  {}%         Indent amount (empty = no indent, \parindent = para indent)
  {\bfseries}% Thm head font
  {.}%        Punctuation after thm head
  {\newline}% Space after thm head: \newline = linebreak
  {}%         Thm head spec

\newtheoremstyle{break1}% name
  {9pt}%      Space above, empty = `usual value'
  {9pt}%      Space below
  {\rmfamily}% Body font
  {}%         Indent amount (empty = no indent, \parindent = para indent)
  {\scshape}% Thm head font
  {.}%        Punctuation after thm head
  {\newline}% Space after thm head: \newline = linebreak
  {}%         Thm head spec

\theoremstyle{break}

\newtheoremstyle{note}% name
  {3pt}%      Space above
  {3pt}%      Space below
  {}%         Body font
  {}%         Indent amount (empty = no indent, \parindent = para indent)
  {\itshape}% Thm head font
  {:}%        Punctuation after thm head
  {.5em}%     Space after thm head: " " = normal interword space;
  {\newline}  %       \newline = linebreak
  {}%         Thm head spec (can be left empty, meaning `normal')

\theoremstyle{note}

\theoremstyle{definition}

\theoremstyle{break1}
\newtheorem{remark}{Remark}[section]

\begin{document}
\bibliographystyle{plainnat} 
\pdfbookmark[0]{Simplifying the Kohlberg Criterion on the Nucleolus: A Correct Approach}{tit}

\title{{Simplifying the Kohlberg Criterion on the Nucleolus: A Correct Approach} 
}
\author{{\bfseries Holger I. MEINHARDT} 
~\thanks{Holger I. Meinhardt, Institute of Operations Research, Karlsruhe Institute of Technology (KIT), Englerstr. 11, Building: 11.40, D-76128 Karlsruhe. E-mail: \href{mailto:Holger.Meinhardt@wiwi.uni-karlsruhe.de}{Holger.Meinhardt@wiwi.uni-karlsruhe.de}} 
}
\maketitle

\begin{abstract}
\citet{Nguyen:16b,Nguyen:17} claimed that he has developed a simplifying set of the Kohlberg criteria that involves checking the balancedness of at most $(n-1)$ sets of coalitions. This claim is not true. Analogous to \citet{Nguyen:16}, he has incorrectly applied the indirect proof by $(\phi \Rightarrow \bot) \Leftrightarrow \neg \phi$. He established in his purported proofs of the main results that a truth implies a falsehood. This is a wrong statement and such a hypotheses must be rejected (cf.~\citet{mei:15,mei:16,mei:16b,mei:17}). Executing a logical correct interpretation ought immediately lead him to the conclusion that his proposed algorithms are deficient. In particular, he had to detect that the imposed balancedness requirement on the test condition $(\cup_{j=1}^{k}\,T_{j})$ within his proposed methods cannot be appropriate. Hence, one cannot expect that one of these algorithms makes a correct selection. The supposed algorithms are wrongly designed and cannot be set in any relation with Kohlberg. Therefore, our objections from~\citet{mei:17} are still valid with some modifications. In particular, he failed to work out modified necessary and sufficient conditions of the nucleolus.\\

%[.25em]

\end{abstract}

%%%%%%%%%%%%%%%%%%%%%%%%%%%% Main Document ###########################

\thispagestyle{empty}

\pagestyle{scrheadings}  \ihead{\empty} \chead{Simplifying the Kohlberg Criterion on the Nucleolus: A Correct Approach} \ohead{\empty}

\section{Introduction}
\label{sec:intodKohlb}
In his recently revised work~\citet{Nguyen:17} would have us believe that only some typos are responsible for erroneous mathematical statements and incorrect applications of propositional logic in~\citet{Nguyen:16b} caused when changing his algorithms from the pre-nucleolus to the nucleolus (cf.~\citet[p.~14]{Nguyen:17}), instead of, a deep misunderstanding of logic and a halting practice of conducting mathematical proofs riddled with wrong statements as well as incorrect manipulations of mathematical objects. In contrast, he supposed that it is enough to invoke on superior authorities while mentioning that his applied indirect proof technique is well-established in the literature (cf.~\citet[p.~14]{Nguyen:17}) and that for these reasons his proofs are correct. Overstating this line of argumentation, it is sufficient that a superior authority tells us that $1+1=3$ is correct in order to apply this result. Science is not easy as that. It is on refutation but certainly not on verification of well-established hypotheses. In addition, it is about to keep a sceptical distance against research subjects rather than loosing scientific distance and objectivity. Moreover, if he was right with his statement that his misguided proof technique was widely applied in the literature, we must have deep concerns on the reliability of the published results.

In detail, we cannot observe any improvements in the revised version of~\citet{Nguyen:17} in comparison to~\citet{Nguyen:16b}. The manuscript is still full of mistakes and logical flawed. The author was even not be able to eliminate all typos from the previous one and added a bunch of new ones. With all goodwill, we cannot find any argument that could be used in favor of the author. We have still to read that a coalition $S$ is in the span of a collection of sets $T$, i.e., $S \in \text{span}(T)$ rather than $\mathbf{1}_{S} \in \text{span}(\{\mathbf{1}_{R}\,\arrowvert\, R \in T\})$. Or that a collection of characteristic vectors of $H_{k-1}$ can be unified with a collection of coalitions $T_{k}$. In addition, we had to find out that Caratheodory Theorem is still incorrectly applied or that obvious misguided logic was not eliminated from the text. However, we notice that the author tried to give a correct definition of the Kohlberg criteria, nevertheless, he failed in the same vein as to give a correct test condition of his supposed algorithms. It is immediately evident that the changed test condition $\cup_{j=1}^{k} T_{j}$ is defective. Thus, the purported algorithms will still make incorrect selections. It ought to be clear that all of this has nothing to do with typos. Hence, our objections presented in~\citet{mei:17} are valid with some modifications. 

Apart from his lack of understanding that the indirect proof by $(\phi \Rightarrow \bot) \Leftrightarrow \neg \phi$ was not applied correctly, we have even to assert that his problem is more fundamental than that, he does not have any understanding what to prove. This can be observed by the fact that he treats the collection of sets $\cup_{j=1}^{k} T_{j}$ derived from an application of the criteria of Kohlberg in the same vein as a collection of sets $\cup_{j=1}^{r} \tilde{T}_{j}$ derived by his simplified approach. For the author both collections are equal. This is obviously not the case. The collection of $\cup_{j=1}^{r} \tilde{T}_{j}$ is a shortening of $\cup_{j=1}^{k} T_{j}$ while eliminating coalitions. In this respect, he failed to prove that only irrelevant coalitions will be deleted from $\cup_{j=1}^{k} T_{j}$, with the consequence that $\cup_{j=1}^{r} \tilde{T}_{j}$ inherits the essential properties of $\cup_{j=1}^{k} T_{j}$. Hence, this caused him to fail to work out a modified Kohlberg characterization, that is, to give modified necessary and sufficient conditions of the nucleolus. Such a result would be a rheme.

The remaining part of this treatise is organized as follows. In Section~\ref{sec:prel} we introduce some basic notations and definitions to investigate the nucleolus as well as the pre-nucleolus. We deviate in our treatise from the notations applied in~\citet{Nguyen:16b,Nguyen:17} which have to be proven inappropriate of presenting a simplifying Kohlberg criterion. Inter alia, this oversimplification of notations had led the author to wrong conclusions and had disguised a possible re-characterization of the (pre-)nucleolus. After having introduced these building blocks, we discuss in Section~\ref{sec:sKohlb} the Kohlberg criteria of the nucleolus as well as of the pre-nucleolus. This allows us to work out of how we have to modify the necessary and sufficient conditions of the pre-nucleolus which are indispensable of getting reliable algorithms for verifying if a pre-imputation is indeed the pre-nucleolus. We focus on the pre-nucleolus to reduce the complexity of our arguments only to its substantial and crucial parts. Notice that we give no proof for this modified Characterization~\ref{thm:modkohlb} of the pre-nucleolus. This characterization illuminates of how we have to restate the pre-nucleolus in comparison to Kohlberg to get a new verification algorithm. Having worked out this re-characterization of the pre-nucleolus, we were able to formulate an alternative and usable but not necessarily reliable Algorithm~\ref{alg:modpnkohl1}. In this respect, we identify some obstacles which probably make it impossible to prove this new characterization. Though we have some evidence that necessity might work, we expect that sufficiency is the more problematical case to prove. A few final remarks close this treatise in Section~\ref{sec:remKohlb}.

\section{Some Preliminaries}
\label{sec:prel}

As we have discussed in the introduction,~\citet{Nguyen:16b,Nguyen:17} confounds the collection of sets obtained by Kohlberg's criteria with those obtained by his simplifying set criteria which is owed among other things by his simplified notational approach. In order to avoid that we misinterpret these sets or that we unify sets of characteristic vectors with that of coalitions, we need to reintroduce a more complex notation and machinery. In the sequel, we rely only on the notation and definitions of~\citet{Nguyen:16b,Nguyen:17} to make clear how deficient the author has handled mathematical objects. However, for the understanding of our arguments this is not really needed. 

A cooperative game with transferable utility is a pair $\langle N,v \rangle $, where $N$ is the non-empty finite player set $N := \{1,2, \ldots, n\}$, and $v$ is the characteristic function $v: 2^{N} \rightarrow \mathbb{R}$ with $v(\emptyset):=0$. A player $i$ is an element of $N$, and a coalition $S$ is an element of the power set of $2^{N}$. The real number $v(S) \in \mathbb{R}$ is called the value or worth of a coalition $S \in 2^{N}$. Let $S$ be a coalition, the number of members in $S$ will be denoted by $s:=|S|$. In addition, we identify a cooperative game by the vector $v := (v(S))_{S \subseteq N} \in \mathcal{G}^{n} = \mathbb{R}^{2^{n}}$, if no confusion can arise. 

If $\mathbf{x} \in \mathbb{R}^{n}$, we apply $x(S) := \sum_{k \in S}\, x_{k}$ for every $S \in 2^{N}$ with $x(\emptyset):=0$. The set of vectors $\mathbf{x} \in \mathbb{R}^{n}$ which satisfies the efficiency principle $v(N) = x(N)$ is called the {\bfseries pre-imputation set} and it is defined by 
\begin{equation} 
  \label{eq:pre-imp}
  \mathcal{I}^{*}(N,v):= \left\{\mathbf{x} \in \mathbb{R}^{n} \;\arrowvert\, x(N) = v(N) \right\},
\end{equation} 
where an element $\mathbf{x} \in \mathcal{I}^{*}(N,v)$ is called a pre-imputation. The set of pre-imputations which satisfies in addition the {\bfseries individual rationality property} $x_{k} \ge v(\{k\})$ for all $k \in N$ is called the {\bfseries imputation set} $\mathcal{I}(N,v)$. 

Given a vector $\mathbf{x} \in \mathcal{I}^{*}(N,v)$, we define the {\bfseries excess} of coalition $S$ with respect to the pre-imputation $\mathbf{x}$ in the game $\langle N,v \rangle $ by 
\begin{equation} 
  \label{eq:exc} 
  e^{v}(S,\mathbf{x}):= v(S) - x(S). 
\end{equation} 
A non-negative (non-positive) excess of $S$ at $\mathbf{x}$ in the game $\langle N,v \rangle $ represents a gain (loss) to the members of the coalition $S$ unless the members of $S$ do not accept the payoff distribution $\mathbf{x}$ by forming their own coalition which guarantees $v(S)$ instead of $x(S)$. 

In order to define the nucleolus $\nu(N,v)$ of a game $v \in \mathcal{G}^{n}$, take any $\mathbf{x} \in \mathcal{I}(N,v)$ to define a $2^{n}-1$-tuple vector $\theta(\mathbf{x})$ whose components are the excesses $e^{v}(S,\mathbf{x})$ of the $2^{n}-1$ non-empty coalitions $\emptyset \neq S \subseteq N$, arranged in decreasing order, that is,
\begin{equation}
 \label{eq:compl_vec}
  \theta_{i}(\mathbf{x}):=e^{v}(S_{i},\mathbf{x}) \ge e^{v}(S_{j},\mathbf{x}) =:\theta_{j}(\mathbf{x}) \qquad\text{if}\qquad 1 \le i \le j \le 2^{n}-1.
\end{equation}
Ordering the so-called complaint or dissatisfaction vectors $\theta(\mathbf{x})$ for all $\mathbf{x} \in \mathcal{I}(N,v)$ by the lexicographic order $\le_{L}$ on $\mathbb{R}^{2^n-1}$, we shall write
\begin{equation}
 \theta(\mathbf{x}) <_{L} \theta(\mathbf{z}) \qquad\text{if}\;\exists\;\text{an integer}\; 1 \le k \le 2^{n}-1,
\end{equation}
such that $\theta_{i}(\mathbf{x}) = \theta_{i}(\mathbf{z})$ for $1 \le i < k$ and $\theta_{k}(\mathbf{x}) < \theta_{k}(\mathbf{z})$. Furthermore, we write $\theta(\mathbf{x}) \le_{L} \theta(\mathbf{z})$ if either $\theta(\mathbf{x}) <_{L} \theta(\mathbf{z})$ or $\theta(\mathbf{x}) = \theta(\mathbf{z})$. Now the nucleolus $\mathcal{N}(N,v)$ of a game $v \in \mathcal{G}^{n}$ over the set $\mathcal{I}(N,v)$ is defined as 
\begin{equation}
 \label{eq:nuc_sol}
  \mathcal{N}(N,v) = \left\{\mathbf{x} \in \mathcal{I}(N,v)\; \arrowvert\; \theta(\mathbf{x}) \le_{L} \theta(\mathbf{z}) \;\forall\; \mathbf{z} \in \mathcal{I}(N,v) \right\}.
\end{equation}
At this set the total complaint $\theta(\mathbf{x})$ is lexicographically minimized over the non-empty compact convex imputation set $\mathcal{I}(N,v)$.~\citet{schm:69} proved that the nucleolus $\mathcal{N}(N,v)$ is non-empty whenever $\mathcal{I}(N,v)$ is non-empty and it consists of a unique point, which is referred to as $\nu(N,v)$. 

Similar, the pre-nucleolus $\mathcal{PrN}(N,v)$ over the pre-imputations set $\mathcal{I}^{*}(N,v)$ is defined by 
\begin{equation}
 \label{eq:prn_sol}
  \mathcal{PrN}(N,v) = \left\{\mathbf{x} \in \mathcal{I}^{*}(N,v)\; \arrowvert\; \theta(\mathbf{x}) \le_{L} \theta(\mathbf{z}) \;\forall\; \mathbf{z} \in \mathcal{I}^{*}(N,v) \right\}.
\end{equation}
The {\bfseries pre-nucleolus} of any game $v \in \mathcal{G}^{n}$ is non-empty as well as unique, and it is referred to as $\nu^{*}(N,v)$. 

\section{Modified Kohlberg Criterion: A Correct Approach}
\label{sec:sKohlb}

In the course of this section we limit ourselves on the investigation of the pre-nucleolus to reduce the complexity of our arguments only to its essential parts. The nucleolus is only discussed to make the links to the~\citet{Nguyen:17} paper as well as to accentuate the differences to the pre-nucleolus. Reducing the complexity of our arguments reveals that the damage of the~\citet{Nguyen:17} results are more fundamental and are not only restricted to misguided logic or notational typos. It is related to a complete ignorance and misunderstanding of what one has to prove. As we shall see in a moment, the author has to give and to establish a modified characterization of Kohlberg's theorem. He failed to give such a characterization, since he treated the collection of sets by his simplifying set approach in the same manner as the collection of sets obtained by Kohlberg's criteria. Thus, he has not proved that his approach only eliminates the irrelevant coalitions from a collection of sets without destroying the original property of a larger collection of coalitions, i.e., changing a balanced collection to an unbalanced and vice versa. Moreover, he also failed to show the reverse argument, i.e., that whenever Kohlberg's properties are satisfied on his simplified collection of sets, then the underlying imputation must be the nucleolus.     

To start with, we need to define the collection of coalitions $\mathcal{D}(N,v;\psi,\mathbf{x})$ having excess larger than or equal to $\psi \in \mathbb{R}$ at the pre-imputation $\mathbf{x} \in \mathcal{I}^{*}(N,v)$. This is accomplished by Definition~\ref{def:dset} which we introduce next. 

\begin{definition}
  \label{def:dset}
For every $\mathbf{x} \in \mathcal{I}^{*}(N,v)$, and $\psi \in \mathbb{R}$ define the set
  \begin{equation}
    \label{eq:dset}
  \mathcal{D}(N,v;\psi,\mathbf{x}) := \left\{S \subseteq N\backslash\{\emptyset\} \,\arrowvert\, e^{v}(S,\mathbf{x}) \ge \psi \right\},
  \end{equation}
\end{definition}
\noindent and let $\mathcal{B}=\{S_{1},\ldots, S_{m}\}$ be a collection of non-empty sets of $N$. We denote the collection $\mathcal{B}$ as balanced whenever there exist positive numbers $w_{S}$ for all $S \in \mathcal{B}$ such that we have $\sum_{S \in \mathcal{B}}\, w_{S}\mathbf{1}_{S} = \mathbf{1}_{N}$. The numbers $w_{S}$ are called weights for the balanced collection $\mathcal{B}$ and $\mathbf{1}_{S}$ is the {\bfseries indicator function} or {\bfseries characteristic vector} $\mathbf{1}_{S}:N \mapsto \{0,1\}$ given by $\mathbf{1}_{S}(k):=1$ if $k \in S$, otherwise $\mathbf{1}_{S}(k):=0$.

A characterization of the pre-nucleolus in terms of balanced collections is due to~\citet{kohlb:71}.

\begin{theorem}
  \label{thm:kohlb}
 Let $\langle\, N, v\, \rangle$ be a TU-game and let be $\mathbf{x} \in \mathcal{I}^{*}(N,v)$. Then $\mathbf{x} = \nu^{*}(N,v) $ if, and only if, for every $\psi \in \mathbb{R}, \mathcal{D}(N,v;\psi,\mathbf{x})\neq \emptyset$ implies that $\mathcal{D}(N,v;\psi,\mathbf{x})$ is a balanced collection over N. 
\end{theorem}
\begin{proof}
  For a proof see~\citet[pp. 817-821]{msz:13}.
\end{proof}

Notice that balancedness is equivalent to following two conditions:

\begin{definition}
  \label{def:koPr1}
  A vector $\mathbf{x} \in \mathcal{I}^{*}(N,v)$ has Property I w.r.t. to TU-game $\langle\,N, v\,\rangle$, if for all $\psi \in \mathbb{R}$ s.t. $\mathcal{D}(N,v;\psi,\mathbf{x}) \neq \emptyset$,
  \begin{equation*}
    Y(\mathcal{D}(N,v;\psi,\mathbf{x}))=\{\mathbf{y} \in \mathbb{R}^{n} \,\arrowvert\, y(S) \ge 0 \;\forall S \in \mathcal{D}(N,v;\psi,\mathbf{x}), y(N) = 0\},
  \end{equation*}
 implies $y(S) = 0, \forall S \in \mathcal{D}(N,v;\psi,\mathbf{x})$. 
\end{definition}

Whereas the definition of Property II states that

\begin{definition}
 \label{def:koPr2}
  A vector $\mathbf{x} \in \mathcal{I}^{*}(N,v)$ has Property II w.r.t. to TU-game $\langle\,N, v\,\rangle$, if for all $\psi \in \mathbb{R}$ s.t. $\mathcal{C}:=\mathcal{D}(N,v;\psi,\mathbf{x}) \neq \emptyset$, there exists $\mathbf{w} \in 2^{|\mathcal{C}|}$ with $\mathbf{w} \ge \mathbf{0}$ s.t. 
  \begin{equation*}
    \mathbf{1}_{N} = \sum_{S \in \mathcal{C}} w_{S}\mathbf{1}_{S} \qquad\text{and}\qquad w_{S} > 0 \quad\forall\; S \in \mathcal{C}.
  \end{equation*}
\end{definition}

This gives us two alternative characterizations of the pre-nucleolus. This first makes use of Property I and states:

\begin{theorem}
  \label{thm:propI}
Let $\langle\, N, v\, \rangle$ be a TU-game and let be $\mathbf{x} \in \mathcal{I}^{*}(N,v)$. Then $\mathbf{x} = \nu^{*}(N,v) $ if, and only if, $\mathbf{x}$ has Property I.
\end{theorem}
Whereas the second is based on Property II, which is given by:

\begin{theorem}
  \label{thm:propII}
Let $\langle\, N, v\, \rangle$ be a TU-game and let be $\mathbf{x} \in \mathcal{I}^{*}(N,v)$. Then $\mathbf{x} = \nu^{*}(N,v) $ if, and only if, $\mathbf{x}$ has Property II.
\end{theorem}

According to incorrect proofs of necessity related to Theorem~\ref{thm:propI} found in the literature based on misguided logic, cf.~for instance,~\citet[Them.~5.2.6~on~pp.~87-88]{pel_sud:07},~\citet[Them.~19.5~on pp.~275-276]{pet:08},~or~\citet[Them.~19.4~on p.~348]{pet:15}, we give now correct proofs by direct arguments and alternatively by contraposition. Notice in this respect that it is every time preferable to conduct a direct proof but rather an indirect proof. In particular, under the aspect that the former is constructive whereas the latter is not.

\begin{proof}[Direct Proof to Theorem~\ref{thm:propI}:]
Let us assume $\mathbf{x} = \nu^{*}(N,v)$, and let $\psi \in \mathbb{R}$ s.t. $\mathcal{D}(N,v;\psi,\mathbf{x}) \neq \emptyset$ with $Y(\mathcal{D}(N,v;\psi,\mathbf{x}))=\{\mathbf{y} \in \mathbb{R}^{n} \,\arrowvert\, y(S) \ge 0 \;\forall S \in \mathcal{D}(N,v;\psi,\mathbf{x}), y(N) = 0\}$. Define next $\mathbf{z}_{\delta} :=\mathbf{x} + \delta\,\mathbf{y}$ for every $\delta >0$. Since $y(N)=0$, it follows that $\mathbf{z}_{\delta} \in \mathcal{I}^{*}(N,v)$. Choose $\delta^{*}>0$ to be small enough for the inequality 
  \begin{equation}
    \label{eq:pn_exceq0}
  e^{v}(S,\mathbf{z}_{\delta^{*}}) > e^{v}(T,\mathbf{z}_{\delta^{*}}),   
  \end{equation}
to satisfy for all $S \in \mathcal{D}(N,v;\psi,\mathbf{x})$ and all $T \in 2^{N}\backslash\mathcal{D}(N,v;\psi,\mathbf{x})$. Now for every $S \in \mathcal{D}(N,v;\psi,\mathbf{x})$, we obtain
  \begin{equation}
   \label{pn_exceq}
    \begin{split}
    e^{v}(S,\mathbf{z}_{\delta^{*}}) & = v(S) - \big(x(S)+\delta^{*}\,y(S) \big) \\
    & = e^{v}(S,\mathbf{x}) - \delta^{*}\,y(S) \le e^{v}(S,\mathbf{x}).
     \end{split}
  \end{equation}
Thus, from~\eqref{eq:pn_exceq0} and~\eqref{pn_exceq} we have $\theta(\mathbf{z}_{\delta^{*}}) \le_{L} \theta(\mathbf{x})$. In addition, by assumption $\mathbf{x} = \nu^{*}(N,v)$ we also have $\theta(\mathbf{x}) \le_{L} \theta(\mathbf{z}_{\delta^{*}})$. Thus, we get $\theta(\mathbf{x}) =_{L} \theta(\mathbf{z}_{\delta^{*}})$ and by the uniqueness of the pre-nucleolus, we attain $\mathbf{x} = \mathbf{z}_{\delta^{*}}$. This implies $y(S) = 0$ for all $S \in \mathcal{D}(N,v;\psi,\mathbf{x})$. 
\end{proof}

Next we present an alternative proof for Theorem~\ref{thm:propI} by contraposition, that is, we show that if $\neg B$, then $\neg A$ holds. In order to recognize the differences between a direct proof and a proof by contraposition, we restate the arguments in common, and set the changed ones in italic.

\begin{proof}[Proof by Contraposition to Theorem~\ref{thm:propI}:]
For running a proof by contraposition, we have to establish that whenever $\mathbf{x}$ has not Property I, then $\mathbf{x}$ is not the pre-nucleolus, which is equivalent to prove that whenever $\mathbf{x}$ is the pre-nucleolus, then $\mathbf{x}$ has Property I.

For doing so, let $\mathbf{x} \in \mathcal{I}^{*}(N,v)$ and {\itshape assume that $\mathbf{x}$ has not Property I}, i.e., for $\psi \in \mathbb{R}$ such that $\mathcal{D}(N,v;\psi,\mathbf{x}) \neq \emptyset$ with $Y(\mathcal{D}(N,v;\psi,\mathbf{x}))=\{\mathbf{y} \in \mathbb{R}^{n} \,\arrowvert\, y(S) \ge 0 \;\forall S \in \mathcal{D}(N,v;\psi,\mathbf{x}), y(N) = 0\}$ {\itshape there exists $y(S)>0$ for some $S \in \mathcal{D}(N,v;\psi,\mathbf{x})$}. Define next $\mathbf{z}_{\delta} :=\mathbf{x} + \delta\,\mathbf{y}$ for every $\delta >0$. Since $y(N)=0$, it follows that $\mathbf{z}_{\delta} \in \mathcal{I}^{*}(N,v)$. Choose $\delta^{*}>0$ to be small enough for the inequality 
  \begin{equation*}
  e^{v}(S,\mathbf{z}_{\delta^{*}}) > e^{v}(T,\mathbf{z}_{\delta^{*}}),  \tag{\ref{eq:pn_exceq0}}
  \end{equation*}
to satisfy for all $S \in \mathcal{D}(N,v;\psi,\mathbf{x})$ and all $T \in 2^{N}\backslash\mathcal{D}(N,v;\psi,\mathbf{x})$. Now for every $S \in \mathcal{D}(N,v;\psi,\mathbf{x})$, we obtain
    \begin{gather*}
    e^{v}(S,\mathbf{z}_{\delta^{*}})  = v(S) - \big(x(S)+\delta^{*}\,y(S) \big) \\
     = e^{v}(S,\mathbf{x}) - \delta^{*}\,y(S) \le e^{v}(S,\mathbf{x}). \tag{\ref{pn_exceq}}
     \end{gather*}
{\itshape By assumption we have $y(S)>0$ for some $S \in \mathcal{D}(N,v;\psi,\mathbf{x})$}. Thus, from~\eqref{eq:pn_exceq0} and~\eqref{pn_exceq} we have $\theta(\mathbf{z}_{\delta^{*}}) <_{L} \theta(\mathbf{x})$. Hence, $\mathbf{x} \not= \nu^{*}(N,v)$ as required. This concludes the proof. 
\end{proof}

While comparing our second alternative proof with the logical flawed ones given by~\citet{pel_sud:07,pet:08,pet:15}, we realize that these authors confound a proof by contraposition with a proof by contradiction. In addition, notice that from a false conclusion, we can either get a false or a correct implication. For the proofs found in~\citet{pel_sud:07,pet:08,pet:15} the latter holds. This is caused by the fact that the imposition of the assumption that $\mathbf{x}$ is the pre-nucleolus has no effect here, and therefore the elaboration of a contradiction is needless too. In fact, their proofs run through as a proof by contraposition even though the authors conducted it as a proof by contradiction. Despite a wrong application of the indirect proof they are able to come up with a correct result. This illusion will not work out in any case, see for instance~\citet{mei:16b}.

\begin{remark}
The problem with the purported proofs given by these authors is that they conduct the indirect proof
\begin{equation}
  \label{eq:logindp}
(\phi \Rightarrow \bot) \Leftrightarrow \neg \phi,
\end{equation}
as if it is a deduction rather than an implication, where $\phi:= (A \land \neg B)$ as well as $(\neg \phi):=(A \Rightarrow B)$. Hence, they argue the premise $(A \land \neg B)$ is satisfied, from which they deduce that $(\neg A)$ holds, i.e., $\mathbf{x}$ is not the pre-nucleolus, a desired contraction occurs. In order to finally conclude that the premise $(A \land \neg B)$ is false, and therefore $(A \Rightarrow B)$ is fulfilled. Thus, the necessity hypotheses of Theorem~\ref{thm:propI} is confirmed due to a deduction.  

However, they overlooked that $(\phi \Rightarrow \bot)$ is an implication and not a deduction due to the equivalence to $(A \Rightarrow B)$. By the implication $(\phi \Rightarrow \bot)$ one has to infer that a truth implies a falsehood. To see this, note that they assumed the premise $\phi$ is true, then $\neg \phi$ is false, and by the equivalence of~\eqref{eq:logindp} the statement $(\phi \Rightarrow \bot)$ is a wrong statement.\footnote{Notice that in~\eqref{eq:logindp} we have only one degree of freedom.} Therefore, they have in reality established that the necessity hypotheses of~Theorem~\ref{thm:propI} must be denied and not accepted with the consequence that one cannot conclude that the initial premises $(A \land \neg B)$ must be rejected and instead of that $\neg(A \land \neg B) \equiv (\neg A \lor B) \equiv (A \Rightarrow B)$ must hold. Moreover, notice that a conclusion must be obtained within a logical closed and consistent system. In contrast, by the premises $(A \land \neg B)$ one tries to extend the logical system though the contradiction $(\neg A)$ occurs outside the logical consistent system. This implies that one has only established that the premises are inconsistent.

In summary, these authors establish that the assumed pre-nucleolus $\mathbf{x}$ is due to the imposed second assumption and by the construction of vector $\mathbf{z}_{\delta}$ not the pre-nucleolus. The system of premises is inconsistent. As a consequence, the premise that $\mathbf{x}$ is the pre-nucleolus must be denied. In contrast, they conclude in the next step that despite this bad selection of the pre-nucleolus, the hypotheses must be satisfied. This is of course a fallacy. It ought to be obvious that if one makes a bad selection for a premise, a hypotheses must be rejected and cannot be accepted. These authors have not shown what they have intended to prove caused by an incorrect application of the indirect proof by $(\phi \Rightarrow \bot) \Leftrightarrow \neg \phi$.\footnote{Based on private communication with Hans Peters.} \hfill$\Diamond$
\end{remark}

Unfortunately, this and the misuse of the proof by contradiction, for instance, in the purported proof of Theorem 16.12 in~\citet{pet:15} or Theorem 16.11 in~\citet{pet:08}\footnote{Actual,~\citet{pet:08,pet:15} have again confounded a proof by contraposition with a proof by contradiction.} give some evidence to the statement of~\citet[p.~14]{Nguyen:17} that this invalid proof technique is widely applied in the literature. Our deep concerns are not causeless.

We return from our short logic detour while introducing an implementation of Theorem~\ref{thm:kohlb} by Algorithm~\ref{alg:pnkohl1} to verify that a pre-imputation is the pre-nucleolus.

\footnotesize
\begin{algorithm} 
\DontPrintSemicolon
\SetAlgoLined
\KwData{Game $\langle\,N, v\,\rangle$, pre-imputation $\mathbf{x} \in \mathcal{I}^{*}(N,v)$.}
\KwResult{Conclude if $\mathbf{x}$ is the pre-nucleolus.}
\Begin{
  \nlset{1} Initialization: Choose pre-imputation $\mathbf{x}$,\; find $\psi_{0} = max_{S \in 2^{N}\backslash\{\emptyset\}}\,\{v(S)-x(S)\}$ to get $ \mathcal{D}_{0}=\mathcal{D}(N,v;\psi_{0},\mathbf{x})$,\; and set $k=1$.\;
 \While{$rank(\{\mathbf{1}_{S}\,\arrowvert\, S \in \mathcal{D}_{k-1}\})<n$}{
  \If{$\mathcal{D}_{k-1}$ is a balanced collection over $N$}{\nlset{2} Find $\psi_{k} = max_{S \in 2^{N}\backslash\mathcal{D}_{k-1}}\,\{v(S)-x(S)\}$;\; Set $\mathcal{D}_{k} = \mathcal{D}(N,v;\psi_{k},\mathbf{x}),\, k=k+1$ and continue;\; \lElse{\;\nlset{3} Stop the algorithm and conclude that $\mathbf{x}$ is not the pre-nucleolus.}}
        }
 }{\nlset{4} Conclude that $\mathbf{x}$ is the pre-nucleolus.}
 \caption{~Algorithm for verifying if a pre-imputation is the pre-nucleolus}
 \label{alg:pnkohl1}
 \end{algorithm}
\normalsize

Now let us turn to the nucleolus to see how we have to modify these properties. For this purpose, denote by $\mathcal{C}_{0}:= \{\{i\} \;\arrowvert\; i \in N\}$ the collection of coalitions containing only one player. Then we can introduce the criteria of Kohlberg by

\begin{definition}
  \label{def:nckoPr1}
  A vector $\mathbf{x} \in \mathcal{I}(N,v)$ has Property I w.r.t. to TU-game $\langle\,N, v\,\rangle$, if for all $\psi \in \mathbb{R}$ s.t. $\mathcal{D}(N,v;\psi,\mathbf{x}) \neq \emptyset$ with $\mathcal{C}(\psi):=\mathcal{C}_{0} \cup \mathcal{D}(N,v;\psi,\mathbf{x})$,
  \begin{equation*}
    Y(\mathcal{C}(\psi))=\{\mathbf{y} \in \mathbb{R}^{n} \,\arrowvert\, y(S) \ge 0 \;\forall S \in \mathcal{C}(\psi), y(N) = 0\},
  \end{equation*}
 implies $y(S) = 0, \forall S \in \mathcal{D}(N,v;\psi,\mathbf{x})$.
\end{definition}

Similar, we denote the collection $\mathcal{B}$ as weakly balanced whenever there exist non-negative weights $w_{S}$ for all $S \in \mathcal{B}$ such that we have $\sum_{S \in \mathcal{B}}\, w_{S}\mathbf{1}_{S} = \mathbf{1}_{N}$. Then we can state the definition of Property II as 

\begin{definition}
 \label{def:nckoPr2}
  A vector $\mathbf{x} \in \mathcal{I}(N,v)$ has Property II w.r.t. to TU-game $\langle\,N, v\,\rangle$, if for all $\psi \in \mathbb{R}$ s.t. $\mathcal{D}(N,v;\psi,\mathbf{x}) \neq \emptyset$ with $\mathcal{C}(\psi):=\mathcal{C}_{0} \cup \mathcal{D}(N,v;\psi,\mathbf{x})$, there exists $\mathbf{w} \in 2^{|\mathcal{C}(\psi)|}$ with $\mathbf{w} \ge \mathbf{0}$ s.t. 
  \begin{equation*}
    \mathbf{1}_{N} = \sum_{S \in \mathcal{C}(\psi)} w_{S}\mathbf{1}_{S} \qquad\text{and}\qquad w_{S} > 0 \quad\forall\; S \in \mathcal{D}(N,v;\psi,\mathbf{x}).
  \end{equation*}
\end{definition}

This gives us two alternative characterizations of the nucleolus.

\begin{theorem}[\citet{kohlb:71}]
  \label{thm:ncpropI}
Let $\langle\, N, v\, \rangle$ be a TU-game and let be $\mathbf{x} \in \mathcal{I}(N,v)$. Then $\mathbf{x} = \nu(N,v) $ if, and only if, $\mathbf{x}$ has Property I of~Definition~\ref{def:nckoPr1}.
\end{theorem}

\begin{theorem}[\citet{kohlb:71}]
  \label{thm:ncpropII}
Let $\langle\, N, v\, \rangle$ be a TU-game and let be $\mathbf{x} \in \mathcal{I}(N,v)$. Then $\mathbf{x} = \nu(N,v) $ if, and only if, $\mathbf{x}$ has Property II of~Definition~\ref{def:nckoPr2}.
\end{theorem}

For completeness, we introduce now the transcription of Kohlberg's properties into the language of~\cite{Nguyen:17}. Notice that also in this revised version those properties are again misquoted as well the notion of $T_{0}$-balancedness. Therefore, the correct definition of Property I in the notation of~\citeauthor{Nguyen:17} states that

\begin{definition}
  \label{def:koPr1a}
  A collection of sets $\mathcal{C} = (Q_{0},Q_{1},\ldots,Q_{p})$ has Property I if for all $k = 1,\ldots,p$
  \begin{equation*}
    Y(\cup_{j=0}^{k}\,Q_{j})=\{\mathbf{y} \in \mathbb{R}^{n} \,\arrowvert\, y(S) \ge 0 \;\forall S \in \cup_{j=0}^{k}\,Q_{j}, y(N) = 0\},
  \end{equation*}
 implies $y(S) = 0, \forall S \in \cup_{j=1}^{k}\,Q_{j}$. 
\end{definition}

Whereas the definition of Property II states that

\begin{definition}
 \label{def:koPr2b}
  A collection of sets $\mathcal{C} = (Q_{0},Q_{1},\ldots,Q_{p})$ has Property II if for all $k = 1,\ldots,p$ there exists $\omega \in 2^{|\mathcal{C}|}$ with $\omega \ge \mathbf{0}$ s.t. 
  \begin{equation*}
      e(N) = \sum_{S \in \cup_{j=0}^{k}\, Q_{j}} \omega_{S}e(S) \qquad\text{and}\qquad \omega_{S} > 0 \quad\forall\; S \in \bigcup_{j=1}^{k}\,Q_{j}.
  \end{equation*}
\end{definition}

Compare this with the Algorithm~\ref{alg:ngu02} implemented by~\citet[Algorithm 2; p.~7]{Nguyen:17}. Obviously, the test condition $(\cup_{j=1}^{k}\,T_{j})$ is wrong. Apart from the unification of a set of characteristic vectors with that of a set of coalitions, it is evident the algorithm cannot work correctly and makes incorrect selections. 

\footnotesize
\begin{algorithm}
\DontPrintSemicolon
\SetAlgoLined
\KwData{Game $\langle\,N, v\,\rangle$, imputation solution $\mathbf{x}$.}
\KwResult{Conclude if $\mathbf{x}$ is the nucleolus or not.}
\Begin{
  \nlset{1} Initialization: Set $H_{0}=\{e_{N},\emptyset\},T_{0}=\{\{i\},\, i=1,\ldots,n\,: x_{i}=v(\{i\})\}$, and \; $k=1$\;
 \While{$rank(H_{k-1})<n$}{
  \nlset{2} Find $T_{k} = \argmax_{S \notin span(H_{k-1})}\,\{v(S)-x(S)\}$;\;
  \If{$(\cup_{j=1}^{k}\,T_{j})$ is $T_{0}$-balanced}{\nlset{3} Set $H_{k}=H_{k-1} \cup T_{k},\, k=k+1$ and continue;\; \lElse{\;\nlset{4} Stop the algorithm and conclude that $\mathbf{x}$ {\bfseries is not} the nucleolus.}}
        }
 }{\nlset{5} Conclude that $\mathbf{x}$ is the nucleolus.}
 \caption{Simplified Kohlberg Algorithm for verifying if a solution is the nucleolus of a cooperative game~(\citet[Algorithm 2; p.~7]{Nguyen:17})}
 \label{alg:ngu02}
 \end{algorithm}
\normalsize
This approach is not appropriate to select, for instance, a pre-nucleolus if $T_{0}=\{\emptyset\}$ or a nucleolus otherwise. To get a correct approach, choose first $\psi \in \mathbb{R}$ s.t. $\mathcal{C}:=\mathcal{D}(N,v;\psi,\mathbf{x}) \neq \emptyset$, and define 

\begin{equation}
  \label{eq:alpsubmax}
  \tilde{\epsilon}:= \psi - max_{S \in 2^{N}\backslash\mathcal{C}}\,\{v(S)-x(S)\}.
\end{equation}
Then define the collection of coalitions, denoted as $\tilde{\mathcal{D}}(N,v;\psi,\tilde{\epsilon},\mathbf{x})$, having excess larger than or equal to $\psi-\tilde{\epsilon} \in \mathbb{R}$ at the pre-imputation $\mathbf{x} \in \mathcal{I}^{*}(N,v)$ such that the associated characteristic vectors are not in the span of a collection of characteristic vectors induced from $\mathcal{B}$ satisfying $\mathcal{B} \subseteq \mathcal{D}(N,v;\psi,\mathbf{x})$ or more formally given by Definition~\ref{def:moddset} whereas $\mathcal{B}$ has a particular labeling.   

\begin{definition}
  \label{def:moddset}
For every $\mathbf{x} \in \mathcal{I}^{*}(N,v)$, and $\psi \in \mathbb{R}$ with $\tilde{\epsilon}$ as by Formula~\eqref{eq:alpsubmax}, choose an arbitrary set $\hat{\mathcal{D}}(N,v;\psi,\mathbf{x}) \subseteq \mathcal{D}(N,v;\psi,\mathbf{x}) $, then we define the set
\begin{equation}
  \label{eq:moddset}
  \tilde{\mathcal{D}}(N,v;\psi,\tilde{\epsilon},\mathbf{x}):=\big\{S \in \mathcal{D}(N,v;\psi-\tilde{\epsilon},\mathbf{x})\backslash\hat{\mathcal{D}}(N,v;\psi,\mathbf{x})\,\arrowvert\, \mathbf{1}_{S} \not\in \text{span}(\{\mathbf{1}_{T}\,\arrowvert\, T \in \hat{\mathcal{D}}(N,v;\psi,\mathbf{x})\})\big\}.
\end{equation}
\end{definition}

Next we introduce a recursive definition of the set $\hat{\mathcal{D}}(N,v;\psi-\tilde{\epsilon},\mathbf{x})$.  

\begin{definition}
  \label{def:moddset2}
For every $\mathbf{x} \in \mathcal{I}^{*}(N,v)$, and $\psi \in \mathbb{R}$ with $\tilde{\epsilon}$ as by Formula~\eqref{eq:alpsubmax}, and define the set
\begin{equation}
  \label{eq:modset2}
  \hat{\mathcal{D}}(N,v;\psi-\tilde{\epsilon},\mathbf{x}):= \hat{\mathcal{D}}(N,v;\psi,\mathbf{x}) \cup \,\tilde{\mathcal{D}}(N,v;\psi,\tilde{\epsilon},\mathbf{x}).
\end{equation}
\end{definition}

A modified characterization of the pre-nucleolus in terms of balanced collections $\hat{\mathcal{D}}(N,v;\psi,\mathbf{x})$ must then be given by 

\begin{chara}
  \label{thm:modkohlb}
 Let $\langle\, N, v\, \rangle$ be a TU-game and let be $\mathbf{x} \in \mathcal{I}^{*}(N,v)$. Then $\mathbf{x} = \nu^{*}(N,v) $ if, and only if, for every $\psi \in \mathbb{R}$ s.t. $\hat{\mathcal{D}}(N,v;\psi-\tilde{\epsilon},\mathbf{x})\neq \emptyset$ implies that $\hat{\mathcal{D}}(N,v;\psi-\tilde{\epsilon},\mathbf{x})$ is a balanced collection over $N$. 
\end{chara}
\begin{remark}[{\bfseries{Identified Problems}}]
\label{rem:pnkohl1}
\hfill
\begin{labeling}[:]{Sufficiency}
 \item[Sufficiency] If $\mathbf{x}$ is the nucleolus, one has to show that $\mathbf{z}=\mathbf{x}$ is satisfied. For doing so, we have to lexicographically compare $\theta(\mathbf{z}) \in \mathbb{R}^{2^{N}-1}$ with $\theta(\mathbf{x}) \in \mathbb{R}^{2^{N}-1}$. This can be done by Kohlberg's approach, because we have different values of excesses at $\mathbf{x}$ s.t. $\psi_{1} > \psi_{2}, \ldots > \psi_{p}$, then we know that $\mathcal{D}(N,v;\psi_{p},\mathbf{x}) = 2^{N}\backslash\{\emptyset\}$ is related to $\theta(\mathbf{x})$. A similar argument applies to $\mathbf{z}$ s.t. $\mathcal{D}(N,v;\psi_{r},\mathbf{z}) = 2^{N}\backslash\{\emptyset\}$ is related to $\theta(\mathbf{z})$ with $r \le p$. However, due the elimination of coalitions we might have $\hat{\mathcal{D}}(N,v;\psi_{p}-\tilde{\epsilon},\mathbf{x}) \subset \mathcal{D}(N,v;\psi_{p},\mathbf{x})$. Hence, the collection of sets $\mathcal{D}(N,v;\psi_{p},\mathbf{x})$ is related to $\theta(\mathbf{x})$, whereas the set $\hat{\mathcal{D}}(N,v;\psi_{p}-\tilde{\epsilon},\mathbf{x})$ is not. Even when we introduce a modified vector of excess values in descending order denoted by $\tilde{\theta}(\mathbf{x})$, we might not be able to lexicographically compare $\tilde{\theta}(\mathbf{x})$ with $\tilde{\theta}(\mathbf{z})$, because of a different cardinality of vectors or in case that we can, the order has completely changed w.r.t. $\theta$. Hence, we lose important information related to the characterization of the pre-nucleolus.  
 \item[Necessity] Notice that it holds either 
   \begin{equation*}
     \begin{split}
       (i) & \quad\mathcal{D}(N,v;\psi,\mathbf{x}) \subseteq \hat{\mathcal{D}}(N,v;\psi-\tilde{\epsilon},\mathbf{x}) \subseteq \mathcal{D}(N,v;\psi-\tilde{\epsilon},\mathbf{x}),\\
    \text{or}\quad (ii) &\quad \mathcal{D}(N,v;\psi,\mathbf{x}) \not\subseteq \hat{\mathcal{D}}(N,v;\psi-\tilde{\epsilon},\mathbf{x}) \subseteq \mathcal{D}(N,v;\psi-\tilde{\epsilon},\mathbf{x}),\\
    \text{or}\quad (iii) &\quad \hat{\mathcal{D}}(N,v;\psi-\tilde{\epsilon},\mathbf{x}) \subset \mathcal{D}(N,v;\psi,\mathbf{x}). 
     \end{split}
   \end{equation*}
However, it is always true that $\hat{\mathcal{D}}(N,v;\psi-\tilde{\epsilon},\mathbf{x})$ is a balanced collection over $N$ whenever $\mathcal{D}(N,v;\psi,\mathbf{x})$ and $\mathcal{D}(N,v;\psi-\tilde{\epsilon},\mathbf{x})$ are balanced collections? This has not been proven by~\citet{Nguyen:17}. 

If the collections $\mathcal{D}(N,v;\psi,\mathbf{x})$ and $\mathcal{D}(N,v;\psi-\tilde{\epsilon},\mathbf{x})$ induce the same rank, then the statement $(i)$ holds. In contrast, if the ranks are different, then probably not. What is with cases $(ii)$ and $(iii)$? Nevertheless, we might not be able to lexicographically compare $\tilde{\theta}(\mathbf{x})$ with $\tilde{\theta}(\mathbf{z}_{\delta^{*}})$, though we have some evidence that necessity might be work.  \hfill$\Diamond$
 \end{labeling}
\end{remark}

In contrast to the invalid procedure applied by~\citet{Nguyen:17} to give proofs for his Theorem 2 and 5, we realize from Remark~\ref{rem:pnkohl1} that it is not enough to simply transcribe logical flawed arguments from the literature that can be found, for instance, in~\citet[Them.~5.2.6~on~pp.~87-88]{pel_sud:07},~\citet[Them.~19.5~on pp.~275-276]{pet:08},~or~\citet[Them.~19.4~on p.~348]{pet:15}, to conduct a proof for Characterization~\ref{thm:modkohlb}. Such a proof would certainly be logical incorrect. From this, we infer that more sophisticated and elaborated arguments are needed as those given by these authors to clarify if Characterization~\ref{thm:modkohlb} can be transformed into a theorem.

The unproven modified Characterization~\ref{thm:modkohlb} of the pre-nucleolus can be transcribed into pseudo-code as represented, for instance, through~Algorithm~\ref{alg:modpnkohl1} below   

\footnotesize
\begin{algorithm}
\DontPrintSemicolon
\SetAlgoLined
\KwData{Game $\langle\,N, v\,\rangle$, pre-imputation $\mathbf{x} \in \mathcal{I}^{*}(N,v)$.}
\KwResult{Conclude if $\mathbf{x}$ is the pre-nucleolus.}
\Begin{
  \nlset{1} Initialization: Choose pre-imputation $\mathbf{x}$,\; find $\psi_{0} = max_{S \in 2^{N}\backslash\{\emptyset\}}\,\{v(S)-x(S)\}$ to set $ \hat{\mathcal{D}}_{0}=\hat{\mathcal{D}}(N,v;\psi_{0},\mathbf{x}) = \mathcal{D}(N,v;\psi_{0},\mathbf{x})$,\; and set $k=1$.\;
 \While{$rank(\{\mathbf{1}_{S}\,\arrowvert\, S \in \hat{\mathcal{D}}_{k-1}\})<n$}{
  \If{$\hat{\mathcal{D}}_{k-1}$ is a balanced collection over $N$}{\nlset{2} Find $\psi_{k} = max_{S \in 2^{N}\backslash\mathcal{D}(N,v;\psi_{k-1},\mathbf{x})}\,\{v(S)-x(S)\}$; and $\tilde{\epsilon}_{k}$ as by Formula~\eqref{eq:alpsubmax}. \;Set $\hat{\mathcal{D}}_{k} = \hat{\mathcal{D}}(N,v;\psi_{k}-\tilde{\epsilon}_{k},\mathbf{x}),\, k=k+1$ and continue;\; \lElse{\;\nlset{3} Stop the algorithm and conclude that $\mathbf{x}$ is not the pre-nucleolus.}}
        }
 }{\nlset{4} Conclude that $\mathbf{x}$ is the pre-nucleolus.}
 \caption{~Modified Algorithm for verifying if a pre-imputation is the pre-nucleolus}
 \label{alg:modpnkohl1}
  \end{algorithm}
\normalsize

\begin{remark}
 \label{rem:pnkohl2}
  We were talking about earlier that we have found some empirical evidence that necessity is applicable. However, the author failed to give evidence that his erroneous Algorithm~\ref{alg:ngu02} is under the necessary condition more efficient than the Kohlberg Algorithm~\ref{alg:pnkohl1}. Thus, one has to estimate if the computational complexity is lower while checking for each $S \in \mathcal{D}(N,v;\psi_{k}-\tilde{\epsilon},\mathbf{x})\backslash\hat{\mathcal{D}}(N,v;\psi,\mathbf{x})$ that its induced characteristic vector $\mathbf{1}_{S}$ is not in $\text{span}(\{\mathbf{1}_{T}\,\arrowvert\, T \in \hat{\mathcal{D}}(N,v;\psi_{k},\mathbf{x})\})$ rather than checking that $\mathcal{D}(N,v;\psi_{k}-\tilde{\epsilon},\mathbf{x})$ is a balanced collection over $N$ for all $k = 1,\ldots,p$. Or to put it differently, checking the span condition is not more costly than solving larger LPs.  \hfill$\Diamond$
\end{remark}

\section{Concluding Remarks}
\label{sec:remKohlb}
 In contrast to that what is claimed by~\citet{Nguyen:17}, we were not able to observe any improvements in his revised version compared to~\citet{Nguyen:16b}. The author has added to the still persisting severe deficiencies of the manuscript a bunch of new ones, so that the reported results and algorithms are still invalid, and cannot be granted to have any connection with Kohlberg as the author did. In particular, the proofs of his main Theorems 2 and 5 are again logically flawed due to an incorrect application of the proof by contradiction. Even though the author changed the test condition within the supposed algorithms, the imposed new test condition is again flawed, because of a misuse of the Kohlberg properties. We have to restate that our objections presented in~\citet{mei:17} are valid with some modifications.

To summarize,~\citet{Nguyen:17} failed to prove that only irrelevant coalitions will be deleted from the simplifying set of the Kohlberg criteria with the consequence that this collection of sets inherits the essential properties of the collection derived by Kohlberg's criteria. As a consequence, the author failed to work out modified necessary and sufficient conditions of the nucleolus. In this respect, we provide a possible re-characterization of the pre-nucleolus in terms of modifying sets of the Kohlberg criteria, and discuss in some details the expected problems in order to get a clear, consistent and concise proof.

\pagestyle{scrheadings} \chead{\empty}  
\footnotesize
\bibliography{modKohlberg}

\end{document}